\documentclass[english]{amsart}
\usepackage[T1]{fontenc}
\usepackage[latin1]{inputenc}
\usepackage{graphicx}
\usepackage{amssymb}
\usepackage{amsmath}
\usepackage{amsthm}
\usepackage{color}
\usepackage{relsize}
\usepackage[all]{xy}
\usepackage{amsfonts}
\usepackage{mathrsfs}
\usepackage{latexsym}
\usepackage{geometry} 
\usepackage{textcomp}
\usepackage[english]{babel}

\theoremstyle{plain}
\newtheorem*{bigtheo}{Theorem}
\newtheorem{theo}{Theorem}[subsection]
\newtheorem{prop}[theo]{Proposition}
\newtheorem{lemm}[theo]{Lemma}
\newtheorem{coro}[theo]{Corollary}
\newtheorem{hypo}[theo]{Hypothesis}

\theoremstyle{definition}
\newtheorem{defi}[theo]{Definition}
\theoremstyle{remark}
\newtheorem{rema}[theo]{Remark}

\newcommand{\T}{\mathcal{T}}
\newcommand{\A}{\mathcal{A}}
\newcommand{\B}{\mathcal{B}}
\newcommand{\C}{\mathcal{C}}
\newcommand{\E}{\mathcal{E}}
\newcommand{\mF}{\mathbb{F}}
\newcommand{\Q}{\mathbb{Q}}
\newcommand{\F}{\mathcal{F}}
\newcommand{\G}{\mathcal{G}}
\newcommand{\Li}{\mathcal{L}}
\newcommand{\wF}{\widetilde{\F}}
\newcommand{\wG}{\widetilde{\G}}

\newcommand{\dd}{\delta}
\DeclareMathOperator{\Ker}{Ker}
\DeclareMathOperator{\I}{Im}
\DeclareMathOperator{\Spec}{Spec}

\title{Duality, refined partial Hasse invariants and the canonical filtration}
\author{St\'ephane Bijakowski}
\address{Imperial College,
Department of Mathematics,
180 Queen's Gate,
London SW7 2AZ,
UK}
\email{s.bijakowski@imperial.ac.uk}
\keywords{$p$-divisible group, Hasse invariant, duality, $\mu$-ordinary locus, canonical filtration}
\subjclass[2010]{14L05 (primary), 11F55 (secondary)}

\begin{document}

\begin{abstract}
Let $G$ be a $p$-divisible group over the ring of integers of $\mathbb{C}_p$, and assume that it is endowed with an action of the ring of integers of a finite unramified extension $F$ of $\mathbb{Q}_p$. Let us fix the type $\mu$ of this action on the sheaf of differentials $\omega_G$. V. Hernandez, following a construction of Goldring and Nicole, defined partial Hasse invariants for $G$. The product of these invariants is the $\mu$-ordinary Hasse invariant, and it is non-zero if and only if the $p$-divisible group is $\mu$-ordinary (i.e. the Newton polygon is minimal given the type of the action). \\
We show that if the valuation of the $\mu$-ordinary Hasse invariant is small enough, then each of these partial Hasse invariants is a product of other sections, the refined partial Hasse invariants. We also give a condition for the construction of these invariants over an arbitrary scheme of characteristic $p$. We then give a simple, natural and elegant proof of the compatibility with duality for the classical Hasse invariant, and show how to adapt it to the case of the refined partial Hasse invariants. Finally, we show how these invariants allow us to compute the partial degrees of the canonical filtration (if it exists).
\end{abstract}

\maketitle

\tableofcontents

\section*{Introduction}

Let $E$ be an elliptic curve over an algebraically closed field $k$ of characteristic $p$. There are two possibilities for the number of $k$-points of the $p$-torsion of $E$ : it is either $p$ or $1$. In the first case, we say that the elliptic curve is ordinary; otherwise we say that it is supersingular. This condition can also be seen on the group structure of the $p$-torsion of $E$ : if it is a product of a multiplicative group and an \'etale one, then the elliptic curve is ordinary. An equivalent condition is the fact that the Eisenstein series $E_{p-1}$ is non-zero at $E$ (if $p \geq 5$). \\
\indent More generally, if $G$ is a $p$-divisible group over $k$, we say that $G$ is ordinary if its $p$-torsion is the product of a multiplicative part by an \'etale part. One can associate to $G$ several invariants. The first one is the Newton polygon, and $G$ is ordinary if and only if this polygon has slopes $0$ and $1$. The second one is the Hasse invariant $ha(G)$; it is a section of the sheaf $(\det \omega_G)^{p-1}$, where $\omega_G$ is the sheaf of differentials of $G$. The section $ha(G)$ is induced by the map $V : \omega_G \to \omega_G^{(p)}$, where $V$ is the Verschiebung, and the superscript denotes a twist by the Frobenius. Then $G$ is ordinary if and only if the Hasse invariant is non-zero. \\
\indent Assume now that $F$ is an unramified extension of degree $f$ of $\mathbb{Q}_p$, and that $G$ has an action of $O_F$, the ring of integers of $F$. The sheaf $\omega_G$ thus decomposes into $\omega_G = \oplus_{i=1}^f \omega_{G,i}$, where $\omega_{G,i}$ is the subsheaf of $\omega_G$ where $O_F$ acts by $\sigma^i$, where $\sigma$ is the Frobenius. Let $d_i$ be the dimension of $\omega_{G,i}$ for all $1 \leq i \leq f$. If there exists an integer $d$ with $d_i = d$ for all $d$, then the Hasse invariant is the product of partial Hasse invariants $ha_i(G)$. The element $ha_i(G)$ is a section of the invertible sheaf $(\det \omega_{G,i-1})^p (\det \omega_{G,i})^{-1}$, and is induced by the Verschiebung, which decomposes into $V_i : \omega_{G,i} \to \omega_{G,i-1}^{(p)}$. We will refer to this case as the ordinary case. \\

The general case is more involved. Indeed, if the previous hypothesis is not satisfied, the Hasse invariant $ha(G)$ is always $0$ (because at least one of the $V_i$ can never be an isomorphism). There is then an obstruction for the $p$-divisible group $G$ to be ordinary. Looking at the Newton polygon of $G$, one can see that this polygon lies always above a certain polygon depending on the collection of integers $\mu :=(d_i)_i$, the Hodge polygon. If it is an equality, one say that the $p$-divisible group is $\mu$-ordinary. It roughly states that the structure of the $p$-divisible group is the best possible given the constraints of the action of the ring $O_F$. The construction of a Hasse invariant in this situation, i.e. a section of an invertible sheaf such that its non-vanishing is equivalent to the fact that the $p$-divisible group is $\mu$-ordinary, has been initiated by Goldring and Nicole (\cite{GN}). Actually, they construct this invariant over a Shimura variety. A local construction has then been done by Hernandez (\cite{He_ha}). \\
\indent Let us recall the main idea of this construction. The iterated of the Verschiebung $V^f$ induces a map $\omega_{G,i} \to \omega_{G,i}^{(p^f)}$, for all $1 \leq i \leq f$. If $d_i$ is the minimum of the $(d_j)_j$, then there is no obstruction for this map to be an isomorphism, and taking the determinant gives a section $Ha_i(G)$ of $(\det \omega_{G,i})^{p^f-1}$. If it is not the case, the determinant of this map is always $0$. However, using crystalline cohomology, one can lift this map to a ring where $p$ is not a zero divisor. Then one can show that the determinant of the map is divisible by an explicit power of $p$. Making the adequate division gives a section $Ha_i(G)$ of $(\det \omega_{G,i})^{p^f-1}$ for any integer $i$. The product of these sections is then the $\mu$-ordinary Hasse invariant. This construction is actually valid over an arbitrary scheme of characteristic $p$, not just the spectrum of an algebraically closed field. Note that the construction in \cite{GN} is valid over the special fiber of Shimura varieties of PEL type, and that Koskivirta and Wedhorn (\cite{KW}) constructed $\mu$-ordinary Hasse invariants for Shimura varieties of Hodge type.\\

In the ordinary case, the section $Ha_i(G)$ is equal to a product of powers of the sections $ha_j(G)$, $1 \leq j \leq f$. This suggests that the situation is not optimal, and that one should be able to define analogues of the $ha_j(G)$ in general. This is indeed possible if one assumes moreover the existence of a certain filtration on the (contravariant) Dieudonn\'e crystal $\E$ of $G$ evaluated at $k$. Recall that $\E$ is a $k$-vector space of dimension the height of $G$, and that it decomposes into $\E = \oplus_{i=1}^f \E_i$.

\begin{bigtheo}
Let $S$ be a scheme of characteristic $p$, and let $G$ be a $p$-divisible group over $S$. Assume that $G$ has an action of $O_F$, and fix the type of this action. Let $\E$ be the Dieudonn\'e crystal of $G$ evaluated at $S$, and assume the existence of adequate filtrations on the sheaves $\E_i$. Then there exist sections $h_k^{[i]} (G)$ of invertible sheaves for $1 \leq k \leq f$ and $1 \leq i \leq f$, such that
$$Ha_i (G) = h_{i}^{[i]} \cdot (h_{i-1}^{[i]})^p \cdot \dots \cdot (h_{i+1}^{[i]})^{p^{f-1}} \text{.}$$
\end{bigtheo}

We refer to Hypotheses \ref{fil} and \ref{adequate} for the precise definition of the adequate filtrations. Let us just remark that they are canonically defined if the $p$-divisible group $G$ is defined over a perfect field and is $\mu$-ordinary. Indeed, Moonen (\cite{Mo}) showed that in that case $G$ admits a canonical filtration by sub-$p$-divisible groups, giving the desired filtrations on the spaces $\E_i$. If one considers a $p$-divisible group $G$ defined over $O_{\mathbb{C}_p} / p$, close to being $\mu$-ordinary (in the sense that the valuation of the $\mu$-ordinary Hasse invariant is small enough), then it is reasonable to expect that the $p$-torsion of $G$ admits a canonical filtration by finite flat subgroups. There would then exist a filtration on each space $\E_i$ refining the Hodge filtration. The definition of the adequate filtrations is made so that it is compatible with these previous filtrations (possibly after reduction modulo a fractional power of $p$). \\
Let us just make explicit the case $f=2$. The Hodge filtration give subsheaves $\F_1 \subset \E_1$ and $\F_2 \subset \E_2$, locally free of rank respectively $d_1$ and $d_2$, and assume for example that $d_1 < d_2$. The existence of adequate filtrations on $\E_1$ and $\E_2$ amounts to the existence of $\F_1 \subset \F_1^{[2]} \subset \E_1$ and $\F_2^{[1]} \subset \F_2$ such that $\F_1^{[2]}$ is locally a direct factor of rank $h+d_1-d_2$ containing the intersection of $\E_1$ with the image of the Frobenius $F$, and $\F_2^{[1]}$ is locally a direct factor of rank $d_2 - d_1$ included in the kernel of the Verschiebung $V$. The sections $h_1^{[1]}$ and $h_2^{[1]}$ are then induced respectively by the determinant of the maps 
$$V : \F_1 \to (\F_2 / \F_2^{[1]})^{(p)} \qquad V : \F_2 / \F_2^{[1]} \to \F_1^{(p)} \text{.}$$
The sections $h_1^{[2]}$ and $h_2^{[2]}$ are then induced respectively by the determinant of the maps
$$F : (\E_2 / \F_2)^{(p)} \to  \F_1^{[2]} / \F_1 \qquad F : (\F_1^{[2]} / \F_1)^{(p)} \to \E_2 / \F_2 \text{.}$$
$ $

If one considers the usual special fiber of a Shimura variety of type (A), then the existence of adequate filtrations is in general not satisfied. Indeed, Hernandez proved in \cite{He_ha} that the elements $Ha_i (G)$ are irreducible in the generic case. But if one considers a certain closed subscheme of a flag variety (see \cite{EV}), then the hypotheses are satisfied and the sections $Ha_i(G)$ are no longer irreducible. Note that the full flag variety for some Shimura varieties, stratifications on this space, and the construction of generalized Hasse invariants on strata, have recently been studied by Goldring and Koskivirta in \cite{GK}. \\
\indent We have thus constructed $f^2$ refined partial Hasse invariants. Actually, there may be fewer of them~: the section $h_k^{[i]}$ depends only on $k$ and the integer $d_i$. In the ordinary case, we just get the usual partial Hasse invariants. On the other hand, if the elements $d_i$ are pairwise distinct, there are $f^2$ distinct refined partial Hasse invariants. \\
$ $\\
\indent The duality is a very natural phenomenon for $p$-divisible groups. A $p$-divisible group is ordinary if and only if its dual is. Moreover, the Hasse invariant is compatible with the duality. This result is quite expected, and not surprising at all; therefore one could expect a very simple proof of this result. It has been proved in \cite{Fa}, proposition $2$; in \cite{Co} Th. $2.3.5$, it is proved that the Hasse invariants of $G$ and $G^D$ generate the same ideals. However, both proofs are very little natural. Indeed, they first deal with the case where the ordinary locus is dense, and then use a descent argument to prove the general case. \\
One would expect the construction of the refined partial Hasse invariants to be compatible with duality, as is shown in the example where $f=2$. Indeed, in this case, one uses the Verschiebung for some refined partial Hasse invariants, and the Frobenius for the others. This suggests to look at the duality in more details, and I got the following result.

\begin{bigtheo}
There is a simple and natural proof of the compatibility with duality for the Hasse invariant. Moreover, this proof can be extended to the case of the refined partial Hasse invariants.
\end{bigtheo}

We refer to Theorems $\ref{dual1}$ and $\ref{dual2}$ for the precise statements and results. \\
$ $\\
\indent The Hasse invariant plays a central role in the theory of the canonical subgroup. Indeed, let $K$ be a finite extension of $\mathbb{Q}_p$ with ring of integers $O_K$, and let $G$ be a $p$-divisible group defined over $O_K$. One can look at the Hasse invariant of $G \times_{O_K} O_K / p$, and taking its (truncated) valuation gives a well defined rational between $0$ and $1$. Fargues (\cite{Fa}) proved that if this valuation is small enough (and $p \geq 3$), then there exists a canonical subgroup $C$ in the $p$-torsion of $G$. Moreover, one can relate the degree of $C$ to the Hasse invariant. \\
\indent If the $p$-divisible group has an action of $O_F$, in the ordinary case, there is no obstruction for the $p$-divisible group to be ordinary and to have a canonical subgroup. One can then define the partial degrees $(\deg_i C)_{1 \leq i \leq f}$ of $C$, and relate them to the partial Hasse invariants $ha_i(G)$ (see \cite{Bi_can}). In the general case, Hernandez proved in \cite{He_can} (under some assumptions on $p$) that if the valuation of the $\mu$-ordinary Hasse invariant of $G$ is small enough, then the $p$-torsion of $G$ admits a canonical filtration. There are thus several canonical subgroups $(C_i)$, each of them being of height $f d_i$. Actually, Hernandez proved that if the valuation of $Ha_i(G)$ is small enough, then there exists a canonical subgroup $C_i$ of height $f d_i$. He also relates the valuation of $Ha_i(G)$ to a certain linear combination of the partial degrees of $C_i$, but is unable to compute each of these partial degrees. They are in fact related to the refined partial Hasse invariants.

\begin{bigtheo}
Let $G$ be a $p$-divisible group over $O_K$ with an action of $O_F$, and assume that there exist adequate filtrations for $G \times_{O_K} O_K / p$. Let $1 \leq i \leq f$ be an integer, and assume that there exists a canonical subgroup $C$ of height $f d_i$ (in the strong sense of Definition \ref{defi_can}). Then we have for $1 \leq k \leq f$
$$\deg_k C^D = \max(d_i - d_k,0) + v(h_k^{[i]}) \text{.}$$
\end{bigtheo}

We need the existence of adequate filtrations for $G \times_{O_K} O_K / p$, so that the refined partial Hasse invariants over $O_K / p$ can be defined. We prove that such filtrations always exist, and that the valuations of the refined partial Hasse invariants obtained do not depend on any choice, if the valuation of the $\mu$-ordinary Hasse invariant is small enough. \\
The point of view I develop about the canonical filtration is thus different from Hernandez'. Indeed, he proves that if the $\mu$-ordinary Hasse invariant has a valuation small enough, then one can construct explicitly subgroups of $p$-torsion. On the other hand, I give a precise definition of the desired subgroups, and prove properties for these subgroups starting with their definitions. This point of view might seem weaker, but is actually sufficient for applications. Indeed, one can show that on strict neighborhoods of the $\mu$-ordinary locus in some Shimura varieties, canonical filtrations (with my definition) always exist. Note that the understanding of such neighborhoods would be a key step in constructing overconvergent modular forms of any weight for Shimura varieties with empty ordinary locus (see \cite{Bi_mu} for the definition of such overconvergent modular forms of classical weight). \\

\textbf{Relation with works of other authors.} The construction of the $\mu$-ordinary Hasse ivnariant goes back to the work of Goldring and Nicole (\cite{GN}). Their work uses the crystalline cohomology, and is valid for PEL Shimura varieties. Another construction has been done for Shimura varieties of Hodge type by Koskivirta and Wedhorn using the theory of $G$-Zip (\cite{KW}). A purely local construction using crystalline cohomology has then been done by Hernandez in \cite{He_ha}. There has then been works on generalized Hasse invariants associated to stratifications of Shimura varieties. More precisely, one can study the Ekedahl-Oort stratification, and for each open stratum, construct a section on the adherence of it such that its non-zero locus is exactly the open stratum. Such constructions have been done by Boxer (\cite{Bo}) for PEL Shimura varieties and Goldring-Koskivirta for Shimura varieties of Hodge type (\cite{GK_eo}). These results have applications to the construction of Galois representations. Moreover, Goldring and Koskivirta have also studied the stratification in flag spaces, and constructed generalized Hasse invariants in this context (\cite{GK}). \\
These works are then global in nature, since it uses a stratification on some variety. On the other hand, the work in this paper is purely local. The main question I was trying to answer is the following: let $G$ be a $p$-divisible group over the ring of integers of $\mathbb{C}_p$ with an unramified action. What invariants (i.e. reals between $0$ and $1$) can one attach to it ? The papers previously stated only give the partial $\mu$-ordinary Hasse invariants constructed by Goldring-Nicole, Hernandez or Koskivirta-Wedhorn, and this result was not optimal in my opinion. I prove here that one can attach to $G$ some filtrations, which are well defined modulo a certain power of $p$. Using these filtrations, one can then define the refined partial Hasse invariants. To define these objects over an arbitrary scheme of characteristic $p$, one has to assume the existence of adequate filtrations. This condition anounts to look at an explicitly, carefully chosen, closed stratum of a certain flag variety. Actually, once this assumption is made, one could use the work of Goldring and Koskivirta on flag spaces (\cite{GK}) to construct the desired refined partial Hasse invariants. But my construction is really simple (see Definition \ref{cons}), and there is thus no need to use unnecessary complicated arguments. The main point of my construction is to come up with the explicit condition ensuring the existence of the refined partial Hasse invariants, and thus this paper is very little related to \cite{GK}. Moreover, my main application is the canonical filtration, and thus my motivation for this paper is completely different from \cite{GK}.   \\

Let us now talk about the organization of the paper. In the first section, we define adequate filtrations and the refined partial Hasse invariants. We also relate them to the invariants constructed by Hernandez. In the second section, we give a simple and natural proof of the compatibility with duality for the Hasse invariant, and prove this compatibility for the refined partial Hasse invariants. In the third section, we prove the existence of such filtrations for $p$-divisible groups over a valuation ring, and prove an uniqueness result. In the fourth section, we relate these invariants to the partial degrees of the canonical filtration. \\

I would like to thank Valentin Hernandez and Beno\^it Stroh for interesting discussions.

\section{Refined partial Hasse invariants} \label{sec1}

\subsection{Definition}

Let $F$ be a finite unramified extension of $\Q_p$ of degree $f$, $O_F$ its ring of integers and $k=\mathbb{F}_{p^f}$ the residue field. Let $\T$ be the set of embeddings of $F$ into $\overline{\Q_p}$; it is a cyclic group of order $f$ generated by the Frobenius $\sigma$. We will thus identify $\T$ and $\mathbb{Z} / f \mathbb{Z}$. Let $S$ be a $k$-scheme, and $G$ a $p$-divisible group over $S$ of height $h_0$. Let $\E$ be the evaluation of the contravariant Dieudonn\'e crystal of $G$ at $S$ (see \cite{BBM} section $3.3$), it is a locally free sheaf over $S$ of rank $h_0$. The Frobenius and Verschiebung induce morphisms

$$V : \E \to \E^{(p)} \qquad \qquad F : \E^{(p)} \to \E$$
where the superscript denotes a twist by the Frobenius.
Let $\F \subset \E$ be the Hodge filtration; it is a locally free subsheaf of $\E$, and induces the exact sequence (see \cite{BBM} Corollary $3.3.5$)
$$0 \to \omega_G \to \E \to \omega_{G^D}^\vee \to 0 $$
where $\omega_G$ is the sheaf of differentials of $G$, $G^D$ is the Cartier dual of $G$, and $\omega_{G^D}^\vee$ is the dual of the sheaf $\omega_{G^D}$. \\
Assume now that $G$ has an action of $O_F$; the sheaf $\E$ thus decomposes in $\E = \oplus_{i=1}^f \E_i$, with $O_F$ acting on $\E_i$ by $\sigma^i$. The morphisms $F$ and $V$ decompose in
$$V_i : \E_{i} \to \E_{i-1} ^{(p)} \qquad \qquad F_i : \E_{i-1} ^{(p)} \to \E_i$$
for $1 \leq i \leq f$, with $\E_0$ being identified with $\E_f$. The height of $G$ in this case must be a multiple of $f$; let us note $h_0 = f h$. Let $i$ be an integer between $1$ and $f$; the sheaf $\E_i$ is locally free of rank $h$ over $S$. Let $\F_i = \F \cap \E_i$, it is a locally free sheaf over $S$. Let us assume that the rank of this sheaf is constant over $S$, equal to an integer $d_i$. The dimension of $G$ is thus $d=\sum_{i=1}^f d_i$. The subsheaf $\F_i$ induces an exact sequence
$$0 \to \omega_{G,i} \to \E_i \to \omega_{G^D,i}^\vee \to 0 $$
where $\omega_{G,i}$ is the subsheaf of $\omega_G$ on which $O_F$ acts by $\sigma^i$. \\
Let $\wF_i = \Ker V_i$; it is a locally free subsheaf of $\E_i$ and is also equal to the image of $F_i$. Recall the equality $\F_{i-1}^{(p)} = \I V_i = \Ker F_i$ (see \cite{EV} section $3.1$). The applications $F_i$ and $V_i$ thus induce isomorphisms
$$ V_i : \E_i / \wF_i \simeq \F_{i-1}^{(p)}  \qquad F_i : (\E_{i-1} / \F_{i-1})^{(p)} \simeq \wF_i$$
The subsheaf $\wF_i$ will be called the conjugate filtration, and induces the exact sequence
$$0 \to (\omega_{G^D,i-1}^\vee)^{(p)} \to \E_i \to \omega_{G,i-1}^{(p)} \to 0$$
If $\G$ is a subsheaf of $\F_{i-1}$, then $V_i^{-1} (\G^{(p)})$ is a subsheaf of $\E_i$ containing $\wF_i$. Similarly, if $\G$ is a subsheaf of $\E_{i-1}$ containing $\F_{i-1}$, then $F_i(\G^{(p)})$ is a subsheaf of $\wF_i$. One can then see that a filtration on $\E_{i-1}$ refining the Hodge filtration gives a filtration on $\E_i$ refining the conjugate filtration. \\
$ $\\
\indent Let $r$ be the cardinality of the set $\{ d_i, 1 \leq i \leq f \} \cap [1,h-1]$, and let us write $\dd_1 < \dots < \dd_r$ the different elements of this set. Define also $\dd_0=0$, $\dd_{r+1}=h$. For any $1 \leq i \leq f$, there exists a unique integer $0 \leq s(i) \leq r+1$ such that $d_i = \dd_{s(i)}$. \\
We will now make the following hypothesis.

\begin{hypo} \label{fil}
For each $1 \leq i \leq f$, there exists a filtration
$$0 \subset \F_i^{[s(i)-1]} \subset \dots \subset \F_i^{[1]} \subset \F_i \subset \F_i^{[r]} \subset \dots \subset \F_i^{[s(i)+1]} \subset \E_i$$
such that
\begin{itemize}
\item the sheaf $\F_i^{[j]}$ is locally a direct factor of $\E_i$, and is locally free of rank $d_{i}-\dd_j$ for $1 \leq j \leq s(i)-1$.
\item the sheaf $\F_i^{[j]}$ is locally a direct factor of $\E_i$, and is locally free of rank $h+d_{i}-\dd_j$ for $s(i)+1\leq j \leq r$.
\end{itemize}
\end{hypo}

\noindent Let $i$ be an integer between $1$ and $f$. We will also set $\F_i^{[s(i)]} :=0$ and $\F_i^{[0]} := \F_i^{[r+1]} := \F_i$. Note that the sheaves $\F_i / \F_i^{[j]}$ are locally free of rank $\dd_j$ for $0 \leq j \leq s(i)$, and the sheaves $\F_i^{[j]} / \F_i$ are locally free of rank $h-\dd_j$ for $s(i)+1 \leq j \leq r+1$. From our previous remark, these filtrations refining the Hodge filtration induce filtrations refining the conjugate filtration : 
$$ 0 \subset \wF_i^{[r]} \subset \dots \subset \wF_i^{[s(i-1)+1]} \subset \wF_i \subset \wF_i^{[s(i-1)-1]} \subset \dots \subset \wF_i^{[1]} \subset \E_i$$
More precisely, we define $\wF_i^{[j]} := V_i^{-1} ((\F_{i-1}^{[j]})^{(p)})$ for $1 \leq j \leq s(i-1)$, and $\wF_i^{[j]} := F_i ((\F_{i-1}^{[j]})^{(p)})$ for $s(i-1) +1 \leq j \leq r$. We will set $\wF_i^{[0]}:=\E_i$ and $\wF_i^{[r+1]} := 0$. Note that $\wF_i^{[s(i-1)]} = \wF_i$. For each $0 \leq j \leq r+1$, the sheaf $\E_i / \wF_i^{[j]}$ is locally free of rank $\dd_j$. \\
We will also make the following hypothesis.

\begin{hypo} \label{adequate}
For each integer $i$ between $1$ and $f$, we have $\F_i^{[j]} \subset \wF_i^{[j]}$ if $1 \leq j \leq s(i) -1$, and $\wF_i^{[j]} \subset \F_i^{[j]}$ if $s(i)+1 \leq j \leq r$.
\end{hypo}

\noindent The filtrations satisfying these two hypotheses will be called adequate. Before going any further, let us make some remarks.

\begin{rema}
If there exists an integer $0 < d < h$, with $d_i = d$ for all $1 \leq i \leq f$, then $s(i)=r=1$ for all $1 \leq i \leq f$ and the two previous hypotheses are empty. We will refer to this case as the ordinary case.
\end{rema}

\begin{rema}
Assume that $f=2$ , and suppose that $0<d_{1} < d_{2} < h$. The first hypothesis is the existence of filtrations
$$0 \subset \F_{1} \subset \F_{1}^{[2]} \subset \E_{1} \qquad 0 \subset \F_{2}^{[1]} \subset \F_{2} \subset \E_{2}$$
with $\F_{1}^{[2]}$ locally free of rank $h-d_{2}+d_{1}$ and $\F_{2}^{[1]}$ locally free of rank $d_{2} - d_{1}$. The conditions in the second hypothesis are $\F_{1}^{[2]} \supset \wF_{1} = \I F_{1}$ and $\F_{2}^{[1]} \subset \wF_{2} = \Ker V_{2}$. 
\end{rema}

\begin{rema}
Let $1 \leq i \leq f$, and $1 \leq j \leq s(i)-1$; the condition $(\F_i^{[j]})^{(p)} \subset (\wF_i^{[j]})^{(p)}$ is equivalent to the fact that the Verschiebung $V_{i+1}$ sends $\wF_{i+1}^{[j]}$ into $(\wF_i^{[j]})^{(p)}$. Similarly, if $s(i)+1 \leq j \leq r$, the condition $(\wF_i^{[j]})^{(p)} \subset (\F_i^{[j]})^{(p)}$ is equivalent to the fact that the Frobenius $F_{i+1}$ sends $(\wF_{i}^{[j]})^{(p)}$ into $\wF_{i+1}^{[j]}$. \\
Adequate filtrations thus induce refinements of the conjugate filtration stable by the Frobenius and Verschiebung.
\end{rema}

\begin{rema}
The definition of adequate filtrations is made to be compatible with the canonical filtration defined in section \ref{defcanfil}. More precisely, we will show in section \ref{canfil} that the existence of a canonical filtration for the $p$-torsion of $G$ imply the existence of adequate filtrations, justifying the previous definition.
\end{rema}

\begin{defi}
Let $1 \leq i \leq f$. If $1 \leq j \leq s(i)$, we define $\Li_i^{[j]} := \det ( \F_i / \F_i^{[j]})$; if $s(i)+1 \leq j \leq r$, we define $\Li_i^{[j]} := \det(\E_i / \F_i^{[j]}) \otimes \det (\F_i)$.
\end{defi}

\noindent We have thus defined an invertible sheaf $\Li_i^{[j]}$ for all $1 \leq i \leq f$ and $1 \leq j \leq r$. Note that $\Li_i^{[s(i)]} = \det \F_i$.

\begin{prop} \label{iso}
Let $1 \leq i \leq f$ and $1 \leq j \leq r$. We have
$$\det (\E_i / \wF_i^{[j]}) \simeq (\Li_{i-1}^{[j]})^p$$
\end{prop}

\begin{proof}
 Assume that $j \leq s(i-1)$. Then we have $\E_i / \wF_i^{[j]} \simeq (\F_{i-1} / \F_{i-1}^{[j]} )^{(p)}$, hence the result. Suppose now that $j > s(i-1)$. We have
$$\det (\E_i / \wF_i^{[j]}) \simeq \det (\E_i / \wF_i) \otimes \det (\wF_i / \wF_i^{[j]}) \simeq \det (\F_{i-1})^p \otimes \det (\E_{i-1} / \F_{i-1}^{[j]})^p = (\Li_{i-1}^{[j]})^p$$
\end{proof}

\begin{defi} \label{cons}
Let $1 \leq i \leq f$; if $1 \leq j \leq s(i)$ we define the application $H_i^{[j]} : \F_i / \F_i^{[j]} \to \E_i / \wF_i^{[j]}$. If $s(i) < j \leq r$, we define $H_i^{[j]} : \wF_i^{[j]} \to \F_i^{[j]} / \F_i$.
\end{defi}

\noindent Note that these applications are well defined thanks to Hypothesis \ref{adequate} (recall that $\F_i^{[s(i)]} = 0$).

\begin{prop}
Let $1 \leq i \leq f$, and let $1 \leq j \leq r$. The determinant of $H_i^{[j]}$ gives a section $h_i^{[j]} \in H^0(S, (\Li_{i-1}^{[j]})^p (\Li_i^{[j]})^{-1})$.
\end{prop}

\begin{proof}
The result is clear if $j \leq s(i)$. Suppose that $s(i) < j \leq r$; the determinant of $H_i^{[j]}$ gives a section of the invertible sheaf
$$\det (\F_i^{[j]} / \F_i) \otimes \det (\wF_i^{[j]})^{-1} \simeq \det (\E_i) \otimes \det (\E_i / \F_i^{[j]})^{-1} \otimes \det (\F_i)^{-1} \otimes \det (\wF_i^{[j]})^{-1} \simeq \det (\E_i / \wF_i^{[j]}) \otimes (\Li_i^{[j]})^{-1}$$
and this sheaf is isomorphic to $(\Li_{i-1}^{[j]})^p (\Li_i^{[j]})^{-1}$.
\end{proof}

\noindent We will call the elements $(h_i^{[j]})_{i,j}$ the refined partial Hasse invariants. Actually, the section $h_i^{[j]}$ might be expressed as a product of other sections.

\begin{defi}
Let $1 \leq i \leq f$ and $1 \leq j \leq r$. If $j \leq s(i)$, we define $\mathcal{M}_i^{[j]} := \det (\wF_i^{[j-1]} / \wF_i^{[j]}) \otimes \det (\F_i^{[j-1]} / \F_i^{[j]})^{-1}$, and $m_i^{[j]}$ the section of this sheaf induced by the determinant of the map
$$\F_i^{[j-1]} / \F_i^{[j]} \to \wF_i^{[j-1]} / \wF_i^{[j]}$$
If $j > s(i)$, we define $\mathcal{N}_i^{[j]} := \det (\F_i^{[j]} / \F_i^{[j+1]}) \otimes \det (\wF_i^{[j]} / \wF_i^{[j+1]})^{-1}$, and $n_i^{[j]}$ the section of this sheaf induced by the determinant of the map
$$\wF_i^{[j]} / \wF_i^{[j+1]} \to \F_i^{[j]} / \F_i^{[j+1]}$$
We also define $\mathcal{N}_i^{[s(i)]} := \det (\E_i / \F_i^{[s(i)+1]}) \otimes \det (\wF_i^{[s(i)]} / \wF_i^{[s(i)+1]})^{-1}$, and $n_i^{[s(i)]}$ the section of this sheaf induced by the determinant of the map
$$\wF_i^{[s(i)]} / \wF_i^{[s(i)+1]} \to \E_i / \F_i^{[s(i)+1]}$$
\end{defi}

\noindent The next proposition is immediate.

\begin{prop}
Let $1 \leq i \leq f$ and $1 \leq j \leq r$. If $j \leq s(i)$, we have
$$h_i^{[j]} = \prod_{k=1}^j m_i^{[k]}$$
If $j > s(i)$, we have
$$h_i^{[j]} = \prod_{k=j}^r n_i^{[k]}$$
\end{prop}

The definition of the partial Hasse invariants, using a refinement of the Hodge filtration, might seem odd. It is in fact more natural to use the conjugate filtration, but one only gets these invariants to the power $p$.

\begin{prop}
Let $1 \leq i \leq f$ be an integer, and let $1 \leq j \leq r$. If $j \leq s(i)$, then the determinant of the map
$$V_{i+1} : \E_{i+1} / \wF_{i+1}^{[j]} \to (\E_{i} / \wF_{i}^{[j]})^{(p)}$$
gives the section $(h_i^{[j]})^p$. If $j > s(i)$, then the determinant of the map
$$F_{i+1} :  (\wF_{i}^{[j]})^{(p)} \to \wF_{i+1}^{[j]}$$
gives the section $(h_i^{[j]})^p$.
\end{prop}

\begin{proof}
Assume that $j \leq s(i)$. The section $h_i^{[j]}$ is induced by the determinant of the map $\F_i / \F_i^{[j]} \to \E_i / \wF_i^{[j]}$. The section $(h_i^{[j]})^p$ is thus induced by the determinant of the map
$$(\F_i / \F_i^{[j]})^{(p)} \to (\E_i / \wF_i^{[j]})^{(p)}$$
But the Verschiebung induces an isomorphism 
$$V_{i+1} : \E_{i+1} / \wF_{i+1}^{[j]} \simeq (\F_i / \F_i^{[j]})^{(p)}$$
hence the result. The case $j > s(i)$ is similar.
\end{proof}

\subsection{Relation with the $\mu$-ordinary partial Hasse invariants} \label{mu_ha}

Recall that we have a decomposition
$$\omega_G = \bigoplus_{i=1}^f \omega_{G,i}$$
where $\omega_{G,i}$ is locally free of rank $d_i$, for each $1 \leq i \leq f$. Let us call $\mu=(d_i)_{1 \leq i \leq f}$; the $\mu$-ordinary partial Hasse invariant attached to $i$, $Ha_i (G)$, is constructed in \cite{He_ha} and is an element in $H^0 (S, (\det \omega_{G,i})^{p^f-1})$ for each $1 \leq i \leq f$ with $d_i \neq 0$. The product of these invariants is the total $\mu$-ordinary Hasse invariant
$$Ha(G) \in H^0(S, (\det \omega_G)^{p^f-1})$$
Note that this last invariant was constructed for some Shimura varieties in \cite{GN}. We will first recall the definition of the elements $Ha_{i} (G)$; actually, we will use the construction in \cite{Bi_He}. In this article, the authors gave a simpler construction of these invariants, and extended it to the ramified case for $p$-divisible groups with Pappas-Rapoport condition. We will then relate these elements to the sections $h_i^{[j]}$. More precisely, we will show that $Ha_i (G)$ is equal to the product of some powers of $h_{k}^{[s(i)]}$, for $1 \leq k \leq f$ (where $1 \leq i \leq f$ is an element with $d_i \notin \{0,h\}$). 

\begin{prop}
Let $1 \leq i \leq f$, and let $d_{i-1} < d \leq h$ be an integer. Then the map
\begin{align*}
\bigwedge^{d_{i-1}} (\E_i) \otimes \bigwedge^{d - d_{i-1}} \wF_i & \to \bigwedge^d \E_{i-1}^{(p)}  \\
(x_1 \wedge \dots \wedge x_{d_{i-1}}) \otimes (F_i y_1 \wedge \dots \wedge F_i y_{d - d_{i-1}}) & \to (V_i x_1 \wedge \dots \wedge V_i x_{d_{i-1}}) \wedge (y_1 \wedge \dots \wedge y_{d - d_{i-1}})
\end{align*}
is well defined, and factors through the natural surjection $\bigwedge^{d_{i-1}} (\E_i) \otimes \bigwedge^{d - d_{i-1}} \wF_i \twoheadrightarrow \bigwedge^d \E_i$.
\end{prop}

\noindent This proposition follows from \cite{Bi_He} prop. $2.2.10$. In that case, we will call $f_i^d : \bigwedge^d \E_i \to \bigwedge^d \E_{i-1}^{(p)}$ the induced map. If $d \leq d_{i-1}$, we will define $f_i^d$ to be $\bigwedge^d V_i$.

\begin{defi}
Let $1 \leq i \leq f$ with $d_i \neq 0$. The $\mu$-ordinary partial Hasse invariant attached to $i$, $Ha_i (G)$, is the section induced by the map
\begin{displaymath}
\xymatrix{
\bigwedge^{d_i} \F_i \hookrightarrow \bigwedge^{d_i} \E_i \ar[r]^{f_i^{d_i}} & \bigwedge^{d_i} \E_{i-1}^{(p)} \to & \dots \ar[r]^{(f_{i+2}^{d_i})^{(p^{f-2})}} & \bigwedge^{d_i} \E_{i+1}^{(p^{f-1})} \ar[rr]^{\bigwedge^{d_i} V_{i+1}^{(p^{f-1})}} && \bigwedge^{d_i} \F_{i}^{(p^f)}
}
\end{displaymath}
\end{defi}

\noindent The $\mu$-ordinary partial Hasse invariant $Ha_i (G)$, is thus a section of $(\det \F_i)^{p^f-1}$.

\begin{theo} \label{link}
Let $1 \leq i \leq f$ with $d_i \notin \{0,h\}$. Then we have
$$Ha_i (G) = h_{i}^{[s(i)]} \cdot (h_{i-1}^{[s(i)]})^p \cdot \dots \cdot (h_{i+1}^{[s(i)]})^{p^{f-1}}$$
\end{theo}

\begin{proof}
We will prove this statement by giving an alternative description of the maps $h_{k}^{[j]}$. Let $1 \leq k \leq f$, and $1 \leq j \leq r$ be integers. First assume that $j \leq s(k)$. Then the element $h_k^{[j]}$ is induced by the natural map $\det (\F_k / \F_k^{[j]}) \to \det (\E_k / \wF_k^{[j]})$. Suppose now that $j > s(k)$. We have isomorphisms of sheaves
\begin{align*}
\det(\F_k^{[j]} / \F_k) \otimes \det(\wF_k^{[j]})^{-1} & \simeq \det(\E_k) \otimes \det(\E_k / \F_k^{[j]})^{-1} \otimes \det(\F_k)^{-1} \otimes \det(\wF_k^{[j]})^{-1} \\
& \simeq \det(\E_k / \wF_k^{[j]}) \otimes \det(\E_k/\F_k^{[j]})^{-1} \otimes \det(\F_k)^{-1}
\end{align*}
The element $h_k^{[j]}$ is obtained by taking the determinant of the natural map $\wF_k^{[j]} \to \F_k^{[j]} / \F_k$. Using the previous isomorphism, it is also induced by the natural map
$$\det (\E_k / \F_k^{[j]}) \otimes \det (\F_k) \to \det (\E_k / \wF_k^{[j]})$$
We are left to unravel the isomorphism $\det (\E_k / \wF_k^{[j]}) \simeq  (\Li_{k-1}^{[j]})^p$ from Proposition \ref{iso}. Assume that $j \leq s(k-1)$. The isomorphism $\det (\E_k / \wF_k^{[j]}) \simeq \det ((\F_{k-1} / \F_{k-1}^{[j]} )^{(p)})$ is given by $\wedge^{\dd_j} V_k$. \\
Now let us suppose that $j > s(k-1)$. There are isomorphisms 
$$\det (\E_k / \wF_k^{[j]}) \simeq \det(\E_k / \wF_k) \otimes \det(\wF_k / \wF_k^{[j]}) \simeq \det (\F_{k-1}^{(p)}) \otimes \det ((\E_{k-1}/ \F_{k-1}^{[j]})^{(p)})$$
Assume that we work locally, and let $e_1, \dots, e_{\dd_j}$ be a basis of $\E_k / \wF_k^{[j]}$. This can be done such that $e_{d_{k-1}+1}, \dots, e_{\dd_j}$ are in $\wF_k$. Since this space is the image of $F_{k}$, we have $e_l = F_k x_l$ for some $x_l \in \E_{k-1}^{(p)}$, for all $d_{k-1}+1 \leq l \leq \dd_j$. The image of $e_1 \wedge \dots \wedge e_{\dd_j}$ by the previous isomorphism is then
$$(V_k(e_{1}) \wedge \dots \wedge V_k(e_{d_{k-1}}) ) \otimes (x_{d_{k-1}+1} \wedge \dots \wedge x_{\dd_j})$$
This concludes the proof.
\end{proof}

\begin{rema}
If $d_i = h$, then the sheaf $(\det \omega_{G,i})^{p^f-1}$ is trivial, and the element $Ha_i (G)$ is a non-zero section of this trivial sheaf. Indeed, the map $f_i^h$ induces an isomorphism between $\det \E_i$ and $(\det \E_{i-1})^p$ for all $1 \leq i \leq f$.
\end{rema}

\section{Compatibility with duality}

\subsection{The classical Hasse invariant}

Let $S$ be a $\mF_p$-scheme, and $G$ be a $p$-divisible over $S$ of height $h_0$ and dimension $d$, with $0 < d <h_0$. Let us denote $\E$ the evaluation of the contravariant Dieudonn\'e crystal at $S$ (see \cite{BBM} section $3.3$); it is a locally free sheaf over $S$ of rank $h_0$. The Frobenius and the Verschiebung induce maps $F : \E^{(p)} \to \E$ and $V : \E \to \E^{(p)}$. The Hodge filtration gives a subsheaf $\F \subset \E$, locally free of rank $d$, and which induces the exact sequence (see \cite{BBM} corollary $3.3.5$)
$$0 \to \omega_G \to \E \to \omega_{G^D}^\vee \to 0$$
where $\omega_G$ is the conormal sheaf of $G$ along its unit section, $G^D$ is the Cartier dual of $G$, and the notation $\mathcal{G}^\vee$ means the dual of the sheaf $\mathcal{G}$. Moreover, $\F^{(p)} = $ Im $V = $ Ker $F$ (see \cite{EV} section $3.1$). Let $\widetilde{\F}$ denote the subsheaf of $\E$ defined by Ker $V = $ Im $F$ ; we will call $\wF$ the conjugate filtration. The Frobenius and Verschiebung induce isomorphisms

\begin{displaymath}
F : (\E / \F)^{(p)} \simeq \widetilde{\F}      \qquad  \qquad   V : \E / \widetilde{\F} \simeq \F^{(p)}
\end{displaymath}

\noindent The sheaf $\wF$ is thus locally free of rank $h_0 - d$ and induces an exact sequence
$$0 \to (\omega_{G^D}^\vee)^{(p)} \to \E \to \omega_G^{(p)} \to 0$$

\begin{defi}
The Hasse map $Ha(G)$ for $G$ is the map $V : \omega_G \to \omega_G^{(p)}$. Let $\Li_G$ be the invertible sheaf defined by $\det \omega_G$. The determinant of the Hasse map induces a section $ha(G) \in H^0(S, \Li_G^{p-1})$, called the Hasse invariant of $G$.
\end{defi}

\noindent The evaluation of the contravariant Dieudonn\'e crystal of $G^D$ at $S$ is $\E^\vee$, and the Hodge filtration on this space is induced by $\F^\bot$ (see \cite{BBM} section $5.3$). Moreover, the Frobenius and Verschiebung are given respectively by
$$V^\vee : (\E^\vee)^{(p)} \to \E^\vee \qquad \qquad F^\vee : \E^\vee \to (\E^\vee)^{(p)}$$
The map $Ha(G^D)^\vee$ is thus the map $F : (\E / \F)^{(p)} \to \E / \F$. 

\begin{lemm}
Using the previous isomorphisms, the Hasse map $Ha(G)$ is the natural map $\F \to~\E / \wF$, obtained by composing the inclusion of $\F$ in $\E$ with the projection to $\E / \wF$. Similarly, the map $Ha(G^D)^\vee$ is the natural map $\wF \to \E / \F$.
\end{lemm}

\begin{proof}
The Hasse map is by definition the map induced by the Verschiebung $\F \to \F^{(p)}$. Since the inverse of the Verschiebung gives an isomorphism $\F^{(p)} \simeq \E / \widetilde{\F}$, the claim follows. The map $Ha(G^D)^\vee$ is the map $F : (\E / \F)^{(p)} \to \E / \F$. The composition of the inverse of the Frobenius $\wF \simeq (\E / \F)^{(p)}$ with this map gives the natural map $\wF \to \E / \F$.
\end{proof}

Before proving the duality compatibility for the Hasse invariant, let us state a general proposition, which will be useful throughout this section.

\begin{prop} \label{dual}
Let $\mathcal{A}$ be a locally free sheaf of rank $r$ over $S$, and let $0 < s < r$ be an integer. Let $\B \subset \A$ and $\C \subset \A$ be two locally free sheaves of rank respectively $s$ and $r-s$, such that $\A / \B$ and $\A/ \C$ are locally free. Then we have an isomorphism of invertible sheaves
$$\det(\A / \B) \otimes \det(\C)^{-1} \simeq \det(\A / \C) \otimes \det(\B)^{-1}$$
Let $x \in H^0(S, \det(\A / \B) \otimes \det(\C)^{-1})$ be the section corresponding to the determinant of the natural map $\C \to \A / \B$. Then $x$ is mapped to $y$ under the isomorphism
$$H^0(S, \det(\A / \B) \otimes \det(\C)^{-1}) \simeq H^0(S, \det(\A / \C) \otimes \det(\B)^{-1})$$
where $y \in H^0(S, \det(\A / \C) \otimes \det(\B)^{-1})$ is the section corresponding to the determinant of the natural map $\B \to \A / \C$.
\end{prop}

\begin{proof}
We have isomorphisms of invertible sheaves
$$\det(\A) \simeq \det(\A / \B) \otimes \det(\B) \simeq \det(\A / \C) \otimes \det(\C)$$
so the invertible sheaves $\det(\A / \B) \otimes \det(\C)^{-1}$ and $\det(\A / \C) \otimes \det(\B)^{-1}$ are both isomorphic to 
$$\det(\A) \otimes \det(\B)^{-1} \otimes \det(\C)^{-1}$$
We thus have isomorphisms
$$H^0(S, \det(\A / \B) \otimes \det(\C)^{-1}) \simeq H^0(S, \det(\A) \otimes \det(\B)^{-1} \otimes \det(\C)^{-1}) \simeq H^0(S, \det(\A / \C) \otimes \det(\B)^{-1})$$
One then sees that the elements $x$ and $y$ are mapped to the same element in $H^0(S, \det(\A) \otimes~\det(\B)^{-1} \otimes~\det(\C)^{-1})$. Namely, they are mapped to the section induced by the determinant of the map
$$\B \oplus \C \to \A$$
\end{proof}

\noindent By a slight abuse of notation, we will say that $x=y$ under the isomorphism of invertible sheaves $\det(\A / \B) \otimes \det(\C)^{-1} \simeq~\det(\A / \C) \otimes \det(\B)^{-1}$. This proposition allows us to give a simple proof for the compatibility with the duality of the Hasse invariant, obtained in \cite{Fa} proposition $2$.

\begin{theo} \label{dual1}
There is an isomorphism $\Li_G^{p-1} \simeq \Li_{G^D}^{p-1}$. With this isomorphism, one has the equality $ha(G) = ha(G^D)$.
\end{theo}

\begin{proof}
We apply the previous proposition to $\A = \E$, $\B = \F$ and $\C = \wF$. Since $(\det \E / \wF) \otimes~(\det \F)^{-1} \simeq~\Li_G^{p-1}$ and $(\det \E / \F) \otimes (\det \wF)^{-1} \simeq \Li_{G^D}^{p-1}$, the first result follows. \\
From the previous lemma, $ha(G)$ is the section of the invertible sheaf $(\det \E / \wF) \otimes (\det \F)^{-1}$ obtained by taking the determinant of the natural map $\F \to \E / \wF$. Similarly, $ha(G^D)^\vee$ is the section of the invertible sheaf $(\det \E / \F) \otimes (\det \wF)^{-1}$ obtained by taking the determinant of the natural map $\wF \to \E / \F$. Note that $ha(G^D)^\vee$ and $ha(G^D)$ induces the same section under the canonical isomorphism of sheaves
$$(\det \E / \F) \otimes (\det \wF)^{-1} \simeq  (\det \wF^\vee) \otimes (\det (\E / \F)^\vee)^{-1}$$
The second part of the proposition allows us to conclude.
\end{proof}

\subsection{Refined partial Hasse invariants}

Suppose now that $G$ has an action of $O_F$, and that there exist adequate filtrations for $G$. We keep the notations from section $\ref{sec1}$. The goal of this section is to prove the compatibility of the sections $h_i^{[j]}$ with duality. The contravariant Dieudonn\'e crystal of $G^D$ evaluated at $S$ is $\mathcal{E}^\vee$. We have a decomposition $\E^\vee = \oplus_{i=1}^f \E_i^\vee$; the Verschiebung and Frobenius on $\E^\vee$ are given respectively by $F_i^\vee : \E_i^\vee \to (\E_{i-1}^\vee)^{(p)}$ and $V_i^\vee : (\E_{i-1}^\vee)^{(p)} \to \E_i^\vee$ for $1 \leq i \leq f$. Let $i$ be an integer between $1$ and $f$. The Hodge filtration is given by $\F_i^\bot \subset \E_i^\vee$, and the conjugate filtration by $\wF_i^\bot \subset \E_i^\vee$. These two sheaves are locally free of rank respectively $h - d_i=:d_i'$ and $d_{i-1}$. The filtrations on $\E_i$ will induce filtrations on $\E_i^\vee$. Let us start with a lemma.

\begin{lemm}
Let $1 \leq i \leq f$ be an integer. Let $\mathcal{A} \subset \mathcal{F}_{i-1}$ be a locally free sheaf, and let $\widetilde{\A} := V_i^{-1} (\A^{(p)})$. Then $(\widetilde{\A})^\bot = V_i^\vee ( (\A^\bot)^{(p)})$. \\
Let $\mathcal{B} \supset \mathcal{F}_{i-1}$ be a locally free sheaf, and let $\widetilde{\B} := F_i (\B^{(p)})$. Then $(\widetilde{\B})^\bot = (F_i^\vee)^{-1} ( (\B^\bot)^{(p)})$.
\end{lemm}

\begin{proof}
One can work locally and assume that $S = \Spec R$ for some ring $R$, and that all the sheaves are free $R$-modules. Let $f \in \E_i^\vee$. Then $f \in (\widetilde{\A})^\bot$ if and only if $f=0$ when restricted to $\widetilde{\A} = V_i^{-1} (\A^{(p)})$. This is equivalent to the fact that $f = g \circ V_i$, for some $g \in (\E_{i-1}^\vee)^{(p)}$ with $g=0$ when restricted to $\A^{(p)}$. This last condition is equivalent to $g \in (\A^\bot)^{(p)}$. Thus
$$f \in (\widetilde{\A})^\bot \Leftrightarrow f = V_i^\vee (g) \text{ for some } g \in (\A^\bot)^{(p)}$$
The proof of the second assertion is similar.
\end{proof}

\noindent The adequate filtrations give a filtration on $\E_i^\vee$ for each $1 \leq i \leq f$ : 
$$0 \subset (\F_i^{[s(i)+1]})^\bot \subset \dots \subset (\F_i^{[r]})^\bot \subset \F_i^\bot \subset (\F_i^{[1]})^\bot \subset \dots \subset (\F_i^{[s(i)-1]})^\bot \subset \E_i^\vee$$
We also have the filtration 
$$0 \subset (\wF_i^{[1]})^\bot \subset \dots \subset (\wF_i^{[r]})^\bot \subset \E_i^\vee$$
From the previous lemma, these two filtrations are compatible in the sense that $(\wF_i^{[j]})^\bot = (F_i^\vee)^{-1} ({(\F_{i-1}^{[j]})^{\bot}}^{(p)})$ for $1 \leq j \leq s(i-1)$, and $(\wF_i^{[j]})^\bot = V_i^\vee ({(\F_{i-1}^{[j]})^{\bot}}^{(p)})$ for $s(i-1) +1 \leq j \leq r$.

\begin{prop}
The filtrations on $\E_i^\vee$, $1 \leq i \leq f$, are adequate.
\end{prop}

\begin{proof}
Note that the set $\{d_i', 1 \leq i \leq f \} \cap [1,h-1]$ consists in $h-\dd_r < \dots < h-\dd_1$. Let us denote $h-\dd_j$ by $\dd_j'$. Let $1 \leq i \leq f$, and let $s(i)+1 \leq j \leq r$ be an integer. The sheaf $(\F_i^{[j]})^\bot$ is locally free of rank $\dd_j - d_i = d_i'-\dd_j'$. If $1 \leq j \leq s(i)-1$, then the sheaf $(\F_i^{[j]})^\bot$ is locally free of rank $h - d_i +\dd_j = h + d_i'-\dd_j'$. This proves that the first hypothesis is satisfied. \\
If $1 \leq j \leq s(i)-1$, the inclusion $\F_i^{[j]} \subset \wF_i^{[j]}$ implies that $(\wF_i^{[j]})^\bot \subset (\F_i^{[j]})^\bot$. Similarly, one has $(\F_i^{[j]})^\bot \subset (\wF_i^{[j]})^\bot$ for $s(i)+1 \leq j \leq r$. This allows us to check the second hypothesis.
\end{proof}

\noindent For each $1 \leq i \leq f$ and $1 \leq j \leq r$, we thus have invertible sheaves $\Li_{G^D,i}^{[j]}$, and sections $h_{G^D,i}^{[j]}$ of $(\Li_{G^D,i-1}^{[j]})^p (\Li_{G^D,i}^{[j]})^{-1}$.

\begin{theo} \label{dual2}
Let $1 \leq i \leq f$, and $1 \leq j \leq r$ be integers. We have an isomorphism 
$$(\Li_{G^D,i-1}^{[r+1-j]})^p (\Li_{G^D,i}^{[r+1-j]})^{-1} \simeq (\Li_{i-1}^{[j]})^p (\Li_{i}^{[j]})^{-1}$$
Using this isomorphism, one has the equality $h_{G^D,i}^{[r+1-j]} = h_i^{[j]}$.
\end{theo}

\begin{proof}
First assume that $j < s(i)$. The application $H_i^{[j]}$ is the natural map $\F_i / \F_i^{[j]} \to \E_i / \wF_i^{[j]}$. The application $H_{G^D,i}^{[r+1-j]}$ is the natural map $(\wF_i^{[j]})^\bot \to (\F_i^{[j]})^\bot / \F_i^\bot$ and is thus equal to the dual of $H_i^{[j]}$. Since a map and its dual induce the same section of the same sheaf, the result follows. \\
The case $j > s(i)$ is similar (or can be treated by duality). We are left with the case $j=s(i)$. The map $H_i^{[s(i)]}$ is the natural map $\F_i \to \E_i / \wF_i^{[s(i)]}$, and the dual of $H_{G^D,i}^{[r+1-s(i)]}$ is the natural map $\wF_i^{[s(i)]} \to \E_i / \F_i$. The result then follows from Proposition $\ref{dual}$.
\end{proof}

\begin{rema}
In the ordinary case, this last result simply says that the unramified partial Hasse invariants are compatible with duality.
\end{rema}

\begin{rema}
Using section \ref{mu_ha}, one finds that the $\mu$-ordinary partial Hasse invariants are compatible with duality. This gives another, much simpler, proof of this result obtained in \cite{He_ha} section 10.
\end{rema}

We also have the following compatibility for the sections $(m_i^{[j]})_{i,j}$ and $(n_i^{[j]})_{i,j}$. For $1 \leq i \leq f$ and $1 \leq j \leq r+1-s(i)$, we have an invertible sheaf $\mathcal{M}_{G^D,i}^{[j]}$, and a section $m_{G^D,i}^{[j]}$ of this sheaf. If $1 \leq i \leq f$ and $r+1-s(i) \leq j \leq r$, we have an invertible sheaf $\mathcal{N}_{G^D,i}^{[j]}$, and a section $N_{G^D,i}^{[j]}$ of this sheaf.

\begin{prop}
Let $1 \leq i \leq f$, and $1 \leq j \leq r$ be integers. If $j \leq s(i)$, we have an isomorphism $\mathcal{M}_i^{[j]} \simeq \mathcal{N}_{G^D,i}^{[r+1-j]}$; under this isomorphism, one has $m_i^{[j]} = n_{G^D,i}^{[r+1-j]}$. If $j \geq s(i)$, we have an isomorphism $\mathcal{N}_i^{[j]} \simeq \mathcal{M}_{G^D,i}^{[r+1-j]}$; under this isomorphism, one has $n_i^{[j]} = m_{G^D,i}^{[r+1-j]}$.
\end{prop}

\begin{proof}
Let $1 \leq i \leq f$, and $1 \leq j \leq s(i)$. The section $m_i^{[j]}$ is induced by the determinant of the map
$$\F_i^{[j-1]} / \F_i^{[j]} \to \wF_i^{[j-1]} / \wF_i^{[j]}$$
The section $n_{G^D,i}^{[r+1-j]}$ is induced by the determinant of the map
$$(\wF_i^{[j]})^\bot / (\wF_i^{[j-1]})^\bot \to (\F_i^{[j]})^\bot / (\F_i^{[j-1]})^\bot$$
But the dual of this map is precisely the map $\F_i^{[j-1]} / \F_i^{[j]} \to \wF_i^{[j-1]} / \wF_i^{[j]}$. Since a map and its dual induce the same section of the same sheaf, the first part of the proposition follows. The second part is similar.
\end{proof}

\section{Properties of adequate filtrations}

\subsection{Existence}

\indent In this section, we prove that locally, adequate filtrations always exist. Let $S = \Spec R$, where $R$ is a ring of characteristic $p$ such that if $x,y \in R$, then $x$ divides $y$ or $y$ divides $x$ (for example, $R$ can be a valuation ring of characteristic $p$). In particular, $R$ is local. Let $G$ be a $p$-divisible group over $S$, endowed with an action of $O_F$ as in the previous section. The module $\omega_G$ thus decomposes into $\omega_G = \oplus_{i=1}^f \omega_{G,i}$, and let $d_i$ be the rank of $\omega_{G,i}$ for all $1 \leq i \leq f$. We will keep the same notations as in the previous section.

\begin{theo}
There exist adequate filtrations on the spaces $(\E_i)_{1 \leq i \leq f}$.
\end{theo}

\begin{proof}
We will explicitly construct adequate filtrations on the spaces $\E_i$. We will construct by induction the spaces $(\F_i^{[j]})_{1 \leq i \leq f}$, for $1 \leq j \leq r$. Let $j$ be an integer between $1$ and $r$, and assume we have constructed the spaces $\F_i^{[k]}$ for $1 \leq i \leq f$ and $1 \leq k \leq j-1$. We have to construct the space $\F_i^{[j]}$ for $1 \leq i \leq f$. Note that if $s(i)=j$, there is nothing to be done (as $\F_i^{[s(i)]}$ is defined to be $0$). This proves that there is at least one element $i$ for which the space $\F_i^{[j]}$ is constructed. Suppose that the space $\F_{i-1}^{[j]}$ is constructed, and we will construct the space $\F_i^{[j]}$. This will conclude the construction. \\
If $s(i)=j$, there is nothing to be done. Assume that $j < s(i)$. We look for a direct factor $\F_i^{[j]}$ of $\E_i$ of rank $d_i - \dd_j$ such that $\F_i^{[j]} \subset \F_i^{[j-1]} \cap \wF_i^{[j]}$, with $\F_i^{[0]} = \F_i$ by convention. The module $\F_i^{[j-1]}$ is free of rank $d_i -\dd_{j-1}$, and is included in $\wF_i^{[j-1]}$ ; let $e_1, \dots, e_{d_i - \dd_{j-1}}$ be a basis of $\F_i^{[j-1]}$. Since $\wF_i^{[j]}$ is of corank $\dd_j - \dd_{j-1}$ in $\wF_i^{[j-1]}$, it is determined by $\dd_j - \dd_{j-1}$ equations. The module $\F_i^{[j-1]} \cap \wF_i^{[j]}$ is thus determined by $\dd_j - \dd_{j-1}$ equations of the form
$$\sum_{k=1}^{d_i - \dd_{j-1}} x_k e_k = 0$$
Thanks to our assumption on the ring $R$, each of these conditions is either empty, or is implied by a stronger condition of the form
$$e_{l} = \sum_{k \neq l} x_k' e_k$$ 
The space $\F_i^{[j-1]} \cap \wF_i^{[j]}$ contains therefore a subspace which is free of rank $d_i - \dd_j$. We have thus constructed $\F_i^{[j]}$ in this case. \\
Suppose now that $j > s(i)$. Let $\G := \F_i^{[j-1]}$ if $j> s(i) + 1$, and $\G := \E_i$ if $j=s(i)+1$. We look for a direct factor $\F_i^{[j]}$ of $\E_i$ of rank $h+d_i - \dd_j$ such that $\F_i \subset \F_i^{[j]} \subset \G$ and $\wF_i^{[j]} \subset \F_i^{[j]}$. Note that we have $\wF_i^{[j]} \subset \wF_i^{[j-1]} \subset \G$. The condition on $\F_i^{[j]}$ is then
$$\F_i + \wF_i^{[j]} \subset \F_i^{[j]} \subset \G$$
Using the same argument as in the previous case, there is a free module of rank $d_i + h -\dd_j$ inside $\G$ containing both $\F_i$ and $\wF_i^{[j]}$ (the modules $\F_i$ and $\wF_i^{[j]}$ are free of rank respectively $d_i$ and $h - \dd_j$). This allows us to construct the space $\F_i^{[j]}$, and concludes the proof. 
\end{proof}

\subsection{Uniqueness}

We will now prove that the reduction of the adequate filtrations modulo a certain ideal is unique. \\
Let $K$ be a valuation field, which is an extension of $\mathbb{Q}_p$, $v$ be the valuation (normalized by $v(p)=1$), and let $O_K$ be the ring of integers. For all real $w >0$, let us define $\mathfrak{m}_w := \{x \in O_K, v(x) \geq w \}$ and $O_{K,\{w\}} := O_K / \mathfrak{m}_w$. In particular $O_{K,\{1\}} = O_K / p O_K$. If $M$ is a $O_K$-module, and $w>0$ is a real, then we will note $M_{\{w\}} := M \otimes_{O_K} O_{K,\{w\}}$. \\
Let $G$ be a $p$-divisible group defined over $O_{K,\{1\}}$, and assume that $G$ has an action of $O_F$ as before. We will keep the same notations as in the previous sections. Assume the existence of adequate filtrations $(\F_i^{[\bullet]})_{1 \leq i \leq f}$ on the spaces $(\E_i)_{1 \leq i \leq f}$. We define
$$w_{i}^{[j]} := v( h_i^{[j]})  \in [0,1]$$
for $1 \leq i \leq f$ and $1 \leq j \leq r$. More precisely, since $h_i^{[j]}$ is the section of an invertible sheaf on $\Spec O_{K,\{1\}}$, a choice of a trivialization of this sheaf allows us to see $h_i^{[j]}$ as an element of $O_{K,\{1\}}$. The valuation of this element is then independent of the choice made. We will also define 
$$w = \sum_{i=1}^f \sum_{j=1}^r w_i^{[j]}$$
Let us start with a lemma.

\begin{lemm}
Let $({\F_i^{[\bullet]}}')_{1 \leq i \leq f}$ be another adequate filtration on the spaces $(\E_i)_{1 \leq i \leq f}$. Let $1 \leq i \leq f$ and $1 \leq j \leq r$ be integers, and let $w_i^{[j]} < \alpha \leq 1$ be a real. Assume that ${\F_{k,\{\alpha\}}^{[l]}}' = \F_{k,\{\alpha\}}^{[l]}$ for $1 \leq k \leq i$ and $1 \leq l \leq j-1$, and that ${\F_{i-1,\{\alpha\}}^{[j]}}' = \F_{i-1,\{\alpha\}}^{[j]}$. Then
$${\F_{i,\{\alpha-w_i^{[j]}\}}^{[j]}}' = \F_{i,\{\alpha-w_i^{[j]}\}}^{[j]}$$
\end{lemm}

\begin{proof}
If $s(i)=j$, the result is obvious. Assume that $j < s(i)$. The space ${\F_{i,\{\alpha\}}^{[j]}}'$ lies inside the kernel of the map
$$\phi : \F_{i,\{\alpha\}}^{[j-1]} \to \wF_{i,\{\alpha\}}^{[j-1]} / \wF_{i,\{\alpha\}}^{[j]} $$
The matrix of this map is of the form
\begin{displaymath}
\left(\begin{array}{cc}
0 & M
\end{array}\right)
\end{displaymath}
where this matrix is written with respect to the decomposition $\F_{i,\{\alpha\}}^{[j]} \subset \F_{i,\{\alpha\}}^{[j-1]}$. The matrix $M$ is associated to the map $\F_{i,\{\alpha\}}^{[j-1]} / \F_{i,\{\alpha\}}^{[j]} \to \wF_{i,\{\alpha\}}^{[j-1]} / \wF_{i,\{\alpha\}}^{[j]}$. In particular $v(\det M) \leq w_i^{[j]}$. Let 
$x = \begin{pmatrix} x_1 \\ x_2 \end{pmatrix}$ 
be in the kernel of $\phi$. Then $M x_2 = 0$ so the coordinates of $x_2$ are of valuation greater than $\alpha - w_i^{[j]}$. This implies that 
$${\F_{i,\{\alpha-w_i^{[j]}\}}^{[j]}}' = \F_{i,\{\alpha-w_i^{[j]}\}}^{[j]}$$
Suppose now that $j > s(i)$. We keep the notations from the previous section. The space ${\F_{i,\{\alpha\}}^{[j]}}' / \F_{i,\{\alpha\}}$ contains the image of the map
$$\psi : \wF_{i,\{\alpha\}}^{[j]} \to \G_{\{\alpha\}} / \F_{i,\{\alpha\}} $$
The matrix of this map is of the form
\begin{displaymath}
\left(\begin{array}{c}
N \\
0
\end{array}\right)
\end{displaymath}
where this matrix is written with respect to the decomposition $\F_{i,\{\alpha\}}^{[j]}  / \F_{i,\{\alpha\}} \subset \G_{\{\alpha\}} / \F_{i,\{\alpha\}}$. The matrix $N$ is associated to the map $\wF_{i,\{\alpha\}}^{[j]} \to \F_{i,\{\alpha\}}^{[j]} / \F_{i,\{\alpha\}}$. In particular $v(\det N) = w_i^{[j]}$. Now let $(X_1 X_2)$ be the matrix of a basis of ${\F_{i,\{\alpha\}}^{[j]}}'$. Since this space contains the image of $\psi$, there exists a matrix $Z$ with $N=X_1 Z$ and $0= X_2 Z$. This implies that $v(\det Z) \leq w_i^{[j]}$ and that the coefficients of $X_2$ are of valuation greater than $\alpha - w_i^{[j]}$. This implies that 
$${\F_{i,\{\alpha-w_i^{[j]}\}}^{[j]}}' = \F_{i,\{\alpha-w_i^{[j]}\}}^{[j]}$$
\end{proof}

\begin{prop} \label{unique}
Assume that $w < 1$, and let $({\F_i^{[\bullet]}}')_{1 \leq i \leq f}$ be another adequate filtration on the spaces $(\E_i)_{1 \leq i \leq f}$. Then we have
$${\F_{i,\{1-w\}}^{[j]}}' = \F_{i,\{1-w\}}^{[j]}$$
for $1 \leq i \leq f$ and $1 \leq j \leq r$.
\end{prop}

\begin{proof}
We use successively the previous lemma. More precisely, we will prove by induction on $j$ that 
$${\F_{i,\{1-\alpha_j\}}^{[j]}}' = \F_{i,\{1-\alpha_j\}}^{[j]}$$
with $\alpha_j = \sum_i \sum_{k \leq j} w_i^{[k]}$ for all $1 \leq i \leq f$ and $1 \leq j \leq r$. Let $1 \leq k \leq r$, assume that the previous relation is true for $1 \leq i \leq f$ and $1 \leq j \leq k-1$. Let $i$ be an element with $s(i) = k$; then the space $\F_i^{[k]}$ needs not to be defined. From the previous proposition, we have
$${\F_{i+1,\{1-\alpha_{k-1}-w_{i+1}^{[k]}\}}^{[k]}}' = \F_{i+1,\{1-\alpha_{k-1}-w_{i+1}^{[k]}\}}^{[k]}$$
Applying successively the previous result, one gets
$${\F_{i+l,\{1-\alpha_{k-1}-w_{i+1}^{[k]}-\dots - w_{i+l}^{[k]}\}}^{[k]}}' = \F_{i+l,\{1-\alpha_{k-1}-w_{i+1}^{[k]}-\dots - w_{i+l}^{[k]}\}}^{[k]}$$
for all $1 \leq l \leq f-1$. Hence the result.
\end{proof}

\begin{rema}
We could have replaced $w$ by 
$$\sum_{i=1}^f \sum_{j \neq s(i)} w_i^{[j]}$$
Of course the proposition is also valid if we replace $w$ by the valuation of the total $\mu$-ordinary Hasse invariant.
\end{rema}

\begin{coro}
Assume that $w < 1/2$. Then the elements $(w_i^{[j]})_{1 \leq i \leq f, 1 \leq j \leq r}$ do not depend on the choice of adequate filtrations, and are thus invariants of the $p$-divisible group $G$.
\end{coro}

\begin{proof}
Let $({\F_i^{[\bullet]}}')_{1 \leq i \leq f}$ be another adequate filtration on the spaces $(\E_i)_{1 \leq i \leq f}$, and let $({h_i^{[j]}}')_{i,j}$ be the sections computed with these filtrations. Let $1 \leq i \leq f$ and $1 \leq j \leq r$ be integers, and assume that $j \leq s(i)$, the other case being similar (or can be treated by duality). Let ${w_i^{[j]}}'$ be the valuation of ${h_i^{[j]}}'$. The elements $w_i^{[j]}$ and ${w_i^{[j]}}'$ are respectively the valuations of the determinants of
$$\F_i / \F_i^{[j]} \to \E_i / \wF_i^{[j]} \qquad \F_i / {\F_i^{[j]}}' \to \E_i / ({\wF_i^{[j]}})'$$
From the previous proposition, the reduction modulo $1-w$ of these maps are equal. Thus
$$\min(w_{i}^{[j]},1-w) = \min({w_i^{[j]}}',1-w)$$
Since $w_i^{[j]} \leq w < 1-w$, we have ${w_i^{[j]}}' = w_i^{[j]}$. 
\end{proof}

\section{The canonical filtration}

Let $K$ be a valuation field, which is an extension of $\mathbb{Q}_p$. We will keep the same notations as in the previous section, and we will consider a $p$-divisible group $G$ over $O_K$. We will also assume that there are adequate filtrations associated to the group $G \times_{O_K} O_K / p$ (the existence of adequate filtrations follows from the previous section).

\subsection{Raynaud group schemes}

In this section, we recall some results concerning finite flat group schemes over $O_K$ endowed with an action of $O_F$. If $H$ is a finite flat group scheme over $O_K$, its degree has been defined in \cite{Fa_deg} and will be denoted by $\deg H$. When $H$ has moreover an action of $O_F$, one can define the partial degrees of $H$, noted $\deg_i H$, for $1 \leq i \leq f$. We refer to \cite{Bi_can} section $2$ for the definition and properties of the partial degrees. \\
We now recall the structure theorem from Raynaud (\cite{Ray}) concerning finite flat group schemes of height $f$ over $O_K$, of $p$-torsion and with an action of $O_F$.

\begin{prop}
Let $H$ be a finite flat group scheme of height $f$ over $O_K$, of $p$-torsion and with an action of $O_{F}$. Then there exist elements $(a_i,b_i)_{1 \leq i \leq f}$ of $O_K$ such that $a_i b_i = p u$ for all $1 \leq i \leq f$ (where $u$ is a fixed $p$-adic unit), with $H$ isomorphic to the spectrum of 
$$O_K[X_1, \dots, X_f] / (X_{i}^p - a_{i+1} X_{i+1})$$
where we identify $X_{f+1}$ and $X_1$. The dual of the group with parameters $(a_i,b_i)_{1 \leq i \leq f}$ is the one with parameter $(b_i,a_i)_{1 \leq i \leq f}$. Moreover, we have $\omega_{H,i} = O_K / a_i$, and therefore $\deg_i H = v(a_i)$, $\deg_i H^D = v(b_i)$ for all $1 \leq i \leq f$. 
\end{prop}

\noindent We will refer to such group schemes as Raynaud group schemes. We will also need a description of the Dieudonn\'e crystal of these group schemes.

\begin{prop} \label{crys}
Let $G$ be a $p$-divisible group over $O_K$ with an action of $O_F$, and let $H \subset G[p]$ be a $O_F$-stable Raynaud group scheme with parameter $(a_i,b_i)_{1 \leq i \leq f}$. Let $\E_H$ be the (contravariant) Dieudonn\'e crystal of $H \times_{O_K} O_K / p$  evaluated at $O_K/p$; it decomposes into $\E_H = \oplus_{i=1}^f \E_{H,i}$, with each $\E_{H,i}$ a free $O_K / p$-module of rank $1$. Let $i$ be an integer between $1$ and $f$. The Verschiebung $V_i : \E_{H,i} \to \E_{H,i-1}^{(p)}$ sends a generator to an element of valuation $p v(b_{i-1})$; the Frobenius $F_i : \E_{H,i-1}^{(p)} \to \E_{H,i}$ sends a generator to an element of valuation $p v(a_{i-1})$.
\end{prop}

\begin{proof}
Let $\E = \oplus_{i=1}^f \E_i$ be the Dieudonn\'e crystal of $G \times_{O_K} O_K / p$ evaluated at $O_K / p$. Let $\wF_i \subset \E_i$ and $\wF_{H,i} \subset \E_{H,i}$ be the images of the Frobenius. We have $\E_i / \wF_i \simeq \omega_{G,i-1}^{(p)}$ and have surjective maps
$$\E_i \to \E_{H,i} \qquad \omega_{G,i-1}^{(p)} \to \omega_{H,i-1}^{(p)}$$
Thus
$$\E_{H,i} / \wF_{H,i} \simeq \omega_{H,i-1}^{(p)} \simeq O_{K,\{1\}} / a_{i-1}^p$$
This gives the result for the Frobenius. The result for the Verschiebung can be obtained by duality. 
\end{proof}

\subsection{Definition of the canonical filtration} \label{defcanfil}

We will now recall some definitions and properties of the canonical filtration. Recall that $G$ is a $p$-divisible group over $O_K$ with an action of $O_F$. The module $\omega_G$ thus decomposes into $\omega_G = \oplus_{i=1}^f \omega_{G,i}$, where $\omega_{G,i}$ is a free $O_K$-module for all $i$ between $1$ and $f$. Let $d_i$ be the rank of $\omega_{G,i}$ for all $1 \leq i \leq f$. We keep the notations from the section \ref{sec1}; in particular, one has an integer $r$, and integers $\delta_1, \dots, \delta_r$. 

\begin{defi} \label{defi_can}
Let $1 \leq j \leq r$, and let $C \subset G[p]$ be a $O_F$-stable finite flat subgroup of height $ f \delta_j$. We say that $C$ is canonical (of height $f \delta_j$) if 
$$\deg C^D < \sum_{i=1}^f \max(\delta_j-d_i,0) + \frac{1}{2}$$
It is a strong canonical subgroup if we have
$$\deg C^D < \sum_{i=1}^f \max(\delta_j-d_i,0) + \frac{1}{p+1}$$
We say that $G$ admits a (strong) canonical filtration if there exist (strong) canonical subgroups of height $f \delta_j$, for all $1 \leq j \leq r$.
\end{defi}

\noindent Let $C \subset G[p]$ be any $O_F$-stable finite flat subgroup of height $f \delta_j$, where $j$ is an integer between $1$ and $r$. Note that since $\deg_i C \leq \deg_i G[p] = d_i$, we have $\deg_i C^D \geq \delta_j - d_i$ for all $1 \leq i \leq f$. The degree of $C^D$ is thus always greater or equal than $\sum_{i=1}^f \max(\delta_j-d_i,0)$; the subgroup is canonical if this degree is close to that value. \\
In the ordinary case (i.e. if there exists an integer $0 < d < h$ with $d_i = d$ for all $1 \leq i \leq f$), this definition agrees with the definition of the canonical subgroup given in \cite{Bi_can} section $3.1$.

\begin{prop}
Let $j$ be an integer between $1$ and $r$; there exists at most one canonical subgroup of height $f \delta_j$. If $C$ is a canonical subgroup of height $f \delta_j$, then $C^\bot \subset G^D [p]$ is canonical of height $f (h - \delta_j)$. \\
Let $1 \leq j \leq k \leq r$ be integers, and assume that $C_{l}$ is a canonical subgroup of height $f \delta_l$, for each $l \in \{j,k\}$. Then $C_j \subset C_k$.
\end{prop}

\begin{proof}
This is a local analogue of \cite{Bi_mu} Proposition $1.24$ and $1.25$, and the proof is similar.
\end{proof}

\noindent In particular, if $G$ admits a canonical filtration, one has
$$0 \subset C_1 \subset \dots \subset C_r \subset G[p]$$
where $C_j$ is the canonical subgroup of height $f \delta_j$, for $1 \leq j \leq r$. \\
The main result of \cite{He_can} is the existence of a canonical subgroup (of height $f d_i$) if the valuation of $Ha_i(G)$ is sufficiently small. We recall this theorem here. 

\begin{theo}[\cite{He_can} Th\'eor\`eme $6.10$] \label{theo_he}
Let $1 \leq i \leq f$ be an integer with $d_i \notin \{0,h\}$. Assume that $p > 4h$, and that the valuation of $Ha_i (G)$ is strictly less than $1/2$. Then there exists a canonical subgroup $C_i$ of height $f d_i$. Moreover, one has
$$\sum_{k=0}^{f-1} p^k (\deg_{i-k} C^D - \max(d_i - d_{i-k},0) ) = v(Ha_i (G))$$
\end{theo}

\begin{rema}
This theorem is actually valid under weaker conditions, see \cite{He_can} Th\'eor\`eme $6.10$ for the precise statement.
\end{rema}

\subsection{The partial degrees of the canonical filtration} \label{canfil}

In this section, we will relate the refined partial Hasse invariants constructed previously to the partial degrees of the canonical subgroups (if they exist). Recall that we assumed the existence of adequate filtrations for $G \times_{O_K} O_K / p$, and we noted $w_i^{[j]}$ the valuation of the section $h_i^{[j]}$, for $1 \leq i \leq f$ and $1 \leq j \leq r$. \\
First let us start with two lemmas.

\begin{lemm}
Let $1 \leq j \leq r$, and assume that $C$ is a canonical subgroup of height $f \delta_j$. Let $\alpha = \deg C^D - \sum_{i=1}^f \max(\delta_j-d_i,0) $. Then for all $1 \leq i \leq f$, we have
$$\deg_i C^D \leq \max(\delta_j-d_i,0)+\alpha$$
\end{lemm}

\begin{proof}
Let $\varepsilon_i := \deg_i C^D - \max(\delta_j - d_i,0)$ for $1 \leq i \leq f$; this is a non-negative real. Moreover, we have
$$\sum_{i=1}^f \varepsilon_i = \alpha $$
We conclude that $\varepsilon_i \leq \alpha$ for all $1 \leq i \leq f$.
\end{proof}

\begin{lemm}
Let $1 \leq j \leq r$, and assume that $C$ is a canonical subgroup of height $f \delta_j$. Let $\alpha = \deg C^D - \sum_{i=1}^f \max(\delta_j-d_i,0) $, and let $1 \leq i \leq f$ be an integer. If $j \leq s(i)$, then $\omega_{C,i,\{1-\alpha\}}$ is a quotient of $\omega_{G,i,\{1-\alpha\}}$, which is free of rank $\delta_j$ over $O_{K,\{1-\alpha\}}$. \\
If $j > s(i)$, then $\omega_{C^\bot,i,\{1-\alpha\}}$ is a quotient of $\omega_{G^D,i,\{1-\alpha\}}$, which is free of rank $h-\delta_j$ over $O_{K,\{1-\alpha\}}$. 
\end{lemm}

\begin{proof}
By duality, one can assume that $j \leq s(i)$. From the previous lemma, we have
$$\deg_i C \geq \delta_j - \alpha$$
Since $\omega_{C,i}$ is generated by $\delta_j$ elements, by the elementary divisors theorem, there exists elements $x_1, \dots,x_{\delta_j}$ in $O_K$ with valuation less than $1$ such that
$$\omega_{C,i} \simeq \bigoplus_{k=1}^{\delta_j} O_{K} / x_k$$
We thus have $\sum_{k=1}^{\delta_j} v(x_k) = \deg_i C$, so $v(x_k) \geq 1 - \alpha$ for all $1 \leq k \leq \delta_j$. This implies that 
$$\omega_{C,i,\{1 - \alpha\}} \simeq (O_{K,\{1 - \alpha\}})^{\delta_j}$$
\end{proof}

\noindent In particular, the existence of a canonical filtration for $G$ implies that both modules $\omega_{G,\{1 - \beta\}}$ and $\omega_{G^D,\{1 - \beta\}}$ are filtered by free quotients, for some real $\beta$ depending on the canonical subgroups. This gives adequate filtrations on $\E_{i,\{1 - \beta\}}$, $1 \leq i \leq f$. \\
We can now state the main theorem of this section, which relates the partial degrees of a canonical subgroup to the refined partial invariants. Note that in the ordinary case, the partial degrees of the canonical subgroup have been computed in \cite{Bi_can} section $3.2$.

\begin{theo}
Let $1 \leq j \leq r$, and assume that $C$ is a strong canonical subgroup of height $f \delta_j$. Then for all $1 \leq i \leq f$, we have
$$\deg_i C^D = \max(\delta_j-d_i,0) + w_i^{[j]}$$
\end{theo}

\begin{proof}
Let $\alpha = \deg C^D - \sum_{i=1}^f \max(\delta_j-d_i,0)$, and we keep the notations from the previous sections. Let $i$ be an integer between $1$ and $f$, and consider the surjective morphism
$$\E_{i,\{1 - \alpha\}} \to \E_{C,i,\{1 - \alpha\}}$$
Let $\wG_{i,\{1 - \alpha\}}$ be the kernel of this map. This is a free $O_{K,\{1-\alpha\}}$-module of rank $h-\delta_j$. \\
From the previous lemma $\G_{i,\{1-\alpha\}} :=\wG_{i,\{1 - \alpha\}} \cap \F_{i,\{1 - \alpha\}}$ is a free $O_{K,\{1-\alpha\}}$-module of rank $d_i - \delta_j$ if $j \leq s(i)$. Thus $V_{i+1}^{-1} (\G_{i,\{1-\alpha\}})^{(p)}$ is a free $O_{K,\{1-\alpha\}}$-module of rank $h-d_i + d_i - \delta_j = h-\delta_j$ containing $\wG_{i+1,\{1 - \alpha\}}$. Since this last module  also free of rank $h - \delta_j$, this is an equality. Note that one has $\G_{i,\{1 - \alpha\}} = \F_{i,\{1 - \alpha\}} $ and $\wG_{i+1,\{1 - \alpha\}}= \wF_{i+1,\{1-\alpha\}}$ if $s(i)=j$. \\
By duality, if $j > s(i)$, then $\G_{i,\{1-\alpha\}} :=\wG_{i,\{1 - \alpha\}} + \F_{i,\{1 - \alpha\}}$ is a free $O_{K,\{1-\alpha\}}$-module of rank $h - \delta_j + d_i$, and $F_{i+1} (\G_{i,\{1-\alpha\}})^{(p)} = \wG_{i+1,\{1 - \alpha\}}$. \\
Assume now that $j \leq s(i)$. Then the map
$$V_{i+1} : \E_{i+1,\{1 - \alpha\}} / \wG_{i+1,\{1 - \alpha\}} \to (\E_{i,\{1 - \alpha\}} / \wG_{i,\{1 - \alpha\}})^{(p)}$$
can thus be identified with the map
$$V_{i+1} : \E_{C,i+1,\{1 - \alpha\}} \to (\E_{C,i,\{1 - \alpha\}})^{(p)}$$
Moreover, the determinant of this map has a valuation equal to $p \deg_i C^D \leq p \alpha$. This can be proved by filtering the group $C$ by Raynaud group schemes and using Proposition \ref{crys}. But this map can also be identified with the natural map
$$(\F_{i,\{1 - \alpha\}} / \G_{i,\{1-\alpha\}})^{(p)} \to (\E_{i,\{1 - \alpha\}} / \wG_{i,\{1 - \alpha\}})^{(p)}$$
Since $p \alpha < 1 - \alpha$, one deduces that the determinant of the natural map
$$\F_{i,\{1 - \alpha\}} / \G_{i,\{1-\alpha\}} \to \E_{i,\{1 - \alpha\}} / \wG_{i,\{1 - \alpha\}}$$
has a determinant of valuation $\deg_i C^D$. \\
By duality, if $j > s(i)$, the determinant of the map
$$F_{i+1} : (\wG_{i,\{1 - \alpha\}})^{(p)} \to \wG_{i+1,\{1 - \alpha\}}$$
has a valuation equal to $p \deg_i (C^\bot)^D = p (d_i - \delta_j + \deg_i C^D) \leq p \alpha$. This map can be identified with the natural map
$$(\wG_{i,\{1 - \alpha\}})^{(p)} \to (\G_{i,\{1-\alpha\}} / \F_{i,\{1-\alpha\}})^{(p)}$$
Since $p \alpha < 1 - \alpha$, the determinant of the natural map
$$\wG_{i,\{1 - \alpha\}} \to \G_{i,\{1 - \alpha\}} / \F_{i,\{1 - \alpha\}}$$
has valuation equal to $d_i - \delta_j + \deg_i C^D$. \\
We would be able to conclude if the filtrations $(\G_{i,\{1-\alpha\}})_{1 \leq i \leq f}$ and $(\F_{i,\{1-\alpha\}}^{[j]})_{1 \leq i \leq f}$ were equal. In the general case, using the same proof as in Proposition \ref{unique}, one proves that
$$\F_{i,\{1 - 2\alpha\}}^{[j]} = \G_{i,\{1-2\alpha\}}$$
for all $1 \leq i \leq f$. Let $i$ be an integer with $s(i) \leq j$. The map $H_i^{[j]}$ agrees with the map
$$ \phi_i :  \F_{i,\{1-\alpha\}} / \G_{i,\{1-\alpha\}} \to \E_{i,\{1 - \alpha\}} / \wG_{i,\{1-\alpha\}}$$
after reduction modulo $\mathfrak{m}_{1 - 2 \alpha}$. But we have seen that  the valuation of the determinant of $\phi_i$ is $\deg_i C^D$. The reduction modulo $\mathfrak{m}_{1 - 2 \alpha}$ of $H_i^{[j]}$ has thus a determinant of valuation $\deg_i C^D$. Since $\deg_i C^D \leq \alpha < 1 - 2 \alpha$, one gets
$$\deg_i C^D = w_i^{[j]}$$
The case with $s(i) > j$ is similar or can be obtained by duality.
\end{proof}

\noindent Note that the relation in Theorem \ref{theo_he} due to Hernandez then follows from this last result together with Theorem \ref{link}.

\begin{coro}
Assume the existence of a strong canonical filtration
$$0 \subset C_1 \subset \dots \subset C_r \subset G[p]$$
Then for all $1 \leq i \leq f$ and $1 \leq j \leq r$, we have
$$\deg_i C_j^D = \max(\delta_j-d_i,0) + w_i^{[j]}$$
In particular, the elements $(w_i^{[j]})_{1 \leq i \leq f,1 \leq j \leq r}$ are well defined.
\end{coro}

\begin{rema}
There is no assumption on $p$ in this corollary (and in our definition of the canonical subgroups), unlike the result of Hernandez (Theorem $\ref{theo_he}$).
\end{rema}

\begin{rema}
The partial degrees of the graded parts $C_k / C_{k-1}$ are related to the valuations of the invariants $m_i^{[k]}$ and $n_i^{[k]}$. More precisely, one has
$$ \deg_i (C_k / C_{k-1})^D = v(m_i^{[k]})      $$
for $1 \leq i \leq f$ and $1 \leq k \leq s(i)$ (with $C_0 := 0$), and
$$ \deg_i (C_{k+1} / C_{k}) = v(n_i^{[k]})      $$
for $1 \leq i \leq f$ and $s(i) \leq k \leq r$ (with $C_{r+1} := G[p]$).
\end{rema}

\bibliographystyle{amsalpha}

\end{document}